\documentclass[11pt]{article}

\usepackage[T1]{fontenc}
\usepackage[utf8]{inputenc}
\usepackage[british]{babel}
\usepackage[
rm={oldstyle=false,proportional=true},
sf={oldstyle=false,proportional=true},
tt={oldstyle=false,proportional=true}
]{cfr-lm}
\usepackage[babel]{microtype}

\usepackage{authblk}

\usepackage{mathrsfs}
\usepackage{mathtools}
\usepackage{amsfonts}
\usepackage{amssymb}
\usepackage{amsthm}

\usepackage{enumerate}

\usepackage{xcolor}
\usepackage{array}
\usepackage{colortbl}
\usepackage{booktabs}
\usepackage{multirow}

\usepackage{siunitx}
\usepackage{xspace}
\usepackage{float}

\usepackage[misc]{ifsym}

\usepackage[noadjust]{cite}
\usepackage[bookmarksnumbered,hidelinks]{hyperref}
\usepackage[capitalise,noabbrev]{cleveref}

\usepackage{orcidlink}

\newcommand{\cgup}{\mathord\uparrow}
\newcommand{\cgdown}{\mathord\downarrow}
\newcommand{\cgstar}{\mathord{\ast}}
\newcommand{\game}[2]{\{ #1 \mathrel\vert #2 \}}
\newcommand{\os}{\mathbin{:}}

\DeclareMathOperator{\Z}{\mathcal{Z}}

\newtheorem{theorem}{Theorem}[section]
\newtheorem{lemma}[theorem]{Lemma}

\newcommand{\thistheoremname}{}
\newtheorem*{genericthm*}{\thistheoremname}
\newenvironment{namedthm*}[1]
{\renewcommand{\thistheoremname}{#1}%
\begin{genericthm*}}
{\end{genericthm*}}

\theoremstyle{definition}

\newtheorem{conjecture}[theorem]{Conjecture}

\newtheorem{problem}[theorem]{Open Problem}

\newcommand{\ruleset}[1]{\textup{\textsc{#1}}\xspace}
\newcommand{\diplace}{\ruleset{digraph placement}}

\newcounter{authcount}
\newcommand{\authDetails}[3]{%
  \stepcounter{authcount}%
  \author[\arabic{authcount}]{%
    \mbox{#1\,$^{\textrm{\href{mailto:#2}{\Letter}}\,\,%
    \ifx&#3&\else\raisebox{-0.2ex}{\orcidlink{#3}}\,\fi}$}%
  }%
}

\title{Constructing All Birthday 3 Games as Digraphs}
\authDetails{Alexander Clow}{alexander_clow@sfu.ca}{0000-0001-7568-9787}
\authDetails{Alfie Davies}{research@alfied.xyz}{0000-0002-4215-7343}
\authDetails{Neil Anderson McKay}{neil@bluered.ca}{}
\affil[1]{ \small{Department of Mathematics, Simon Fraser University}}
\affil[2]{ \small{Department of Mathematics and Statistics, Memorial University
of Newfoundland}}
\affil[3]{ \small{Department of Mathematics and Statistics, University of New
Brunswick}}
\date{}

\begin{document}
\maketitle

\begin{abstract}
    \noindent Recently, Clow and McKay proved that the \diplace ruleset is
    universal for normal play: for all normal play combinatorial games $X$,
    there is a \diplace game $G$ with $G=X$. Clow and McKay also showed that
    the 22 game values born by day 2 correspond to \diplace games with at most
    4 vertices. This bound is best possible. We extend this work using a
    combination of exhaustive and random searches to demonstrate all 1474
    values born by day 3 correspond to \diplace games on at most $8$ vertices.
    We provide a combinatorial proof that this bound is best possible. We
    conclude by giving improved bounds on the number of vertices required to
    construct all game values born by days 4 and 5.
\end{abstract}

\section{Introduction}

Let $G=(V,E)$ be a directed graph and let $\phi: V \to
\{\text{blue},\text{red}\}$ be a (not necessarily proper) $2$-colouring of $G$.
The \diplace game played on $(G,\phi)$, which we write simply as $G$ when
$\phi$ is clear from context, is described as follows: there are two players
who play alternately, Left (blue) and Right (red); on her turn, Left chooses a
blue vertex $u$ and deletes $N^+[u]$ from $G$; on his turn, Right chooses a red
vertex $v$ and deletes $N^+[v]$ from $G$. The \diplace game resulting from a
player choosing the vertex $v$ (i.e.\ from the deletion of a vertex $v$ and its
out-neighbours) in a \diplace game $G$ is denoted $G/v$. When a sequence of
vertices $v_1,\dots,v_k$ and their out-neighbours are deleted, we write the
resulting game $G/[v_1,\dots, v_k]$. If, at the beginning of a player's turn,
there are no vertices of that player's colour remaining, then that player is
unable to move and thus loses (as is the normal play convention). The set of
all \diplace games is the ruleset \diplace. 

\begin{figure}[ht]
    \centering
    \scalebox{1}{
        \includegraphics{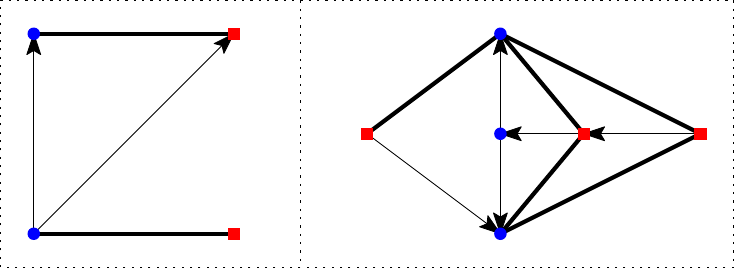}%
    }
    \caption{A \diplace game equal to $\cgup$ (left) and another equal to
        $\game{2}{-2}$ (right). Blue vertices are circles, red vertices are
        squares, and pairs of arcs $(u,v),(v,u)$ are shown with a bolded edge.
    }
    \label{fig:example diplace game}
\end{figure}

\diplace is an example of a (normal play) combinatorial game theory ruleset as
introduced by Conway in \cite{Conwa1976a}. A \emph{normal play} ruleset is one
where there are two players who alternate turns, both players have perfect
information about the game state and moves (the game involves no randomness),
and if on a player's turn they cannot move then they lose. For the rest of the
paper we take \emph{game} to mean an instance of a normal play ruleset. All
games considered have finite numbers of options and end in finitely many turns,
regardless of how the players move.

It is customary to call the players of a game Left and Right. A game $X$ can be
written as $X = \game{X_1,\dots,X_k}{Y_1,\dots,Y_t}$ where
$L(X)\coloneq\{X_1,\dots,X_k\}$ is the set of all games Left can move to from
$X$ (i.e.\ Left's options), and $R(X)\coloneq\{Y_1,\dots,Y_t\}$ is the
analogous set for Right (i.e.\ Right's options). It is a fundamental theorem of
CGT that every game $X$ has one of the following four outcomes: Left wins
(playing first or second), Right wins (playing first or second), the first
player wins, or the second player wins.

\diplace was introduced in 2024 by Clow and McKay \cite{clow2024digraph} as a
means of generalising some well-studied combinatorial game rulesets. Rulesets
that were generalised include, but are not limited to, \ruleset{nim},
\ruleset{col}, \ruleset{snort}, \ruleset{poset games}, \ruleset{node kayles},
\ruleset{arc kayles}, \ruleset{domineering}, and \ruleset{teetering towers}.
For more on these rulesets, see \cite{BerleCG1982,bodlaender2002kayles,
clow2023ordinal, schaefer1978complexity}. Like all of these rulesets, \diplace
is a placement game. Placement games are a special class of rulesets introduced
by Brown et al.\ \cite{brown2019note}. The \diplace ruleset is a natural
instance of a placement game played on digraphs, and it is interesting because,
as Clow and McKay \cite{clow2024digraph} demonstrated, it is the first known
universal ruleset that is a placement game (ruleset). That is, for all normal
play combinatorial games $X$, there is a \diplace game $G$ such that $G=X$.

Two games $X$ and $Y$ are equal, written $X=Y$, if for all games $Z$, the
outcome of $X+Z$ is the same as the outcome of $Y+Z$. The equivalence classes
under this notion of equality are called \emph{values}. Addition here is the
disjunctive sum of games (see \cite{siegel2013combinatorial}).

The birthday of a game $X$, denoted by $b(X)$, is the least (non-negative)
integer strictly greater than the birthday of all the options of $X$. As we
consider finite games, birthdays are well defined. For more on the basic theory
of disjunctive sums and equality we refer the reader to Siegel's textbook
\cite{siegel2013combinatorial}. The birthday of a value is the minimum birthday
of all games with that value. Note, $b(X)$ is sometime called the formal
birthday of $X$ to distinguish the birthday of $X$ from the birthday of the
value of $X$.

We consider the following problem: given a game $X$, what is the least integer
$f(X)$ such that there exists a \diplace game $G=X$ on $f(X)$ vertices. We
focus on bounding the function $F(b)$, which is the maximum of $f(X)$  over all
games $X$ born by day $b$. Clow and McKay \cite{clow2024digraph} proved
$F(2)=4$ and posed determining $F(3)$ as an open problem. The main contribution
of this paper is proving $F(3)=8$.

\begin{theorem}
    \label{Thm: F(3) = 8}
    For all games $X$ born by day 3, there exists a \diplace game $G$ with at
    most 8 vertices such that $G=X$. Furthermore, there exists a game $X$ with
    birthday 3 such that $H\neq X$ for all \diplace games $H$ with at most 7
    vertices.
\end{theorem}

The paper is structured as follows. In \cref{sec:search} we explain our proof
that $F(3)\leq 8$ which uses both exhaustive and random searches. The relevant
output of these searches can be found in the ancillary files and acts as a
certificate that $F(3)\leq 8$. In \cref{sec:8-vertex}, we prove $F(3)\geq8$ via
a combinatorial argument, which demonstrates an explicit construction of a game
requiring 8 vertices. In \cref{sec:better-bounds} we provide improved bounds on
$F(4)$ and $F(5)$.

\section{Constructing all Day 3 Games}
\label{sec:search}

Among its many uses, the \texttt{nauty} package \cite{MCKAY201494} provides a
suite of programs called \texttt{gtools} for generating non-isomorphic graphs.
OEIS sequence A00059 \cite{oeis} reports that there are \num{112282908928}
non-isomorphic digraphs on 7 unlabelled vertices with loops allowed but not
multiple arcs. Using \texttt{gtools}, we generated all of these $7$-vertex
digraphs up to isomorphism, as well as all of those on at most $6$ vertices.

There is a bijection between the set of digraphs on $n$ unlabelled vertices
with loops allowed but not multiple arcs, and the set of (not necessarily
properly) $2$-coloured digraphs on $n$ unlabelled vertices with no loops and no
multiple arcs. Such a bijection can be constructed simply by mapping a vertex
with a loop to a blue vertex, and a vertex without a loop to a red vertex.

As such, with \texttt{gtools}, we generated all \diplace positions that use at
most 7 vertices. To exhaustively check through all of these positions, we wrote
a program in Rust that converted each digraph into a combinatorial game and
kept a hash table of all of the game values found so far. We used the
\texttt{gemau} library \cite{davies:gemau} to work with the combinatorial game
representations of the \diplace positions.

We leveraged Rust's concurrency model to improve performance by using multiple
threads. This efficiency was very beneficial given the scale of our
computations.

The exhaustive search of all \ruleset{diplace} graphs using at most 7 vertices
yielded $1455$ out of the $1474$ values born by day 3. To find the remaining
values, we performed a targeted random sampling of graphs. In particular, we
generated random digraphs on 8 vertices using the Erd\H{o}s--R\'enyi--Gilbert
model with various edge probabilities, focusing primarily around $p=0.5$ (with
colours chosen uniformly at random), and for each we computed its value. We
generated at least 20 billion graphs in this way, but did not keep detailed
logs of these searches. 

\begin{table}[ht]
    \begin{center}
        \begin{tabular}{|l |r | r|} 
            \hline
            \text{Blue-red digraphs searched} & \text{New values found} &
            \text{Instances checked}\\
            [0.5ex]\hline\hline 
            \text{At most $6$ vertices} & \num{1089} & \num{97224121} \\
            [0.5ex] \hline
            \text{$7$-vertex} & \num{366} &  \num{112282908928} \\
            [0.5ex] \hline
            \text{Random $8$-vertex} & 18 & ~\num{20000000000} \\
            [0.5ex] \hline
            \text{Colour Isomorphic $8$-vertex} & 1 & \num{14286848} \\
            [0.5ex] \hline
        \end{tabular}
    \end{center}
    \caption{Summary of the search for representatives of values born by
        day~$3$. Each row describes one aspect of our search, with the first
        rows having been conducted first, and the last rows having been
        conducted last. All values considered are day $3$.
    }
    \label{tab:How we found everything}
\end{table}

Sampling random $8$-vertex digraphs yielded $18$ of $19$ missing day $3$
values. A full list of these $19$ values is given in the appendix. A
significant portion of these values were found early in the random search, with
a majority of the search being spent trying to find a small number of the
missing values. In particular, we were unable to find the value
\[
    \game{\cgup+\cgstar,\cgup,\game{1}{\cgstar,0}}{\game{0,\cgstar}{-1},\cgdown,\cgdown+\cgstar}
\]
via our random search, which we call $\Z$; in \cref{sec:8-vertex}, we prove
that $\Z$ requires at least 8 vertices to represent.

In order to construct a \diplace game on at most 8-vertices equal to $\Z$, we
performed an exhaustive search over those $8$-vertex digraphs exhibiting an
automorphism switching blue and red vertices. We call such graphs \emph{colour
isomorphic}. We generated this list of digraphs as follows: first, we
constructed all 218 isomorphism classes for monochromatic digraphs on
4-vertices with \texttt{gtools}; then, for each of these $218$ isomorphism
classes, we take a representative $H$ and add arcs in all possible ways from a
blue copy of $H$ to a red copy of $H$ (and vice-versa) such that there is an
automorphism which switches the blue and red copies of $H$. 

We treat the copies of $H$ as labelled, and thus it is clear that there are
$2^{16}$ possible ways to configure the arcs joining the blue and red copies of
$H$. This of course overcounts the number of 8-vertex colour isomorphic
digraph, but not by a significant margin. Finally, we checked the value of the
$218\cdot2^{16}=\num{14286848}$  8-vertex colour isomorphic digraphs we
generated. For one example of a digraph equal to $\Z$, see \cref{fig:X}.

The source code for both the exhaustive and random searches are available from
the authors upon request. Thankfully, one does not have to perform the search
again to validate the results, as we have an explicit list of digraphs that
have the 1474 values born by day 3. Given one of these digraphs, it is trivial
to manually check that it has the asserted value.

The digraphs we generated are given in the ancillary \texttt{graphs.d6} file;
their corresponding values are given (in the same order) in the
\texttt{games.txt} file. Digraphs are stored in the \texttt{digraph6} format
where a loop corresponds to a blue vertex and the absence of a loop corresponds
to a red vertex. We provide a short script utilising \texttt{gemau} to automate
the checking that each of the digraphs provided has the game value asserted,
which the reader can find in \texttt{src/main.rs} (with instructions for
running it in \texttt{README.md}).

\section{Some Day 3 Games Require 8 Vertices}
\label{sec:8-vertex}

In \cref{sec:8-vertex}, we demonstrate that there exists a game $X$ with
birthday~$3$ and $f(X)\geq8$. That is, $X$ requires at least 8 vertices to be
realised as a value of a \diplace position. Technically, this was already shown
by our exhaustive search of \diplace games with at most 7 vertices, as we
summarised in \cref{sec:search}. However, reproducing the search requires
significant computer resources, hence this search is not a good certificate
that $F(3)\geq 8$. To this end, we provide a `human' proof that $F(3)\geq8$
which does not use computer assistance. 

Providing multiple avenues to assert $F(3)\geq8$ is beneficial because
demonstrating good lower bounds on $f(X)$ and $F(b)$ has proven difficult. We
feel that our combinatorial argument is easier to generalise to more vertices
than is exhaustive search on at most 7 vertices.

We assume the reader is familiar with the basics of combinatorial game theory.
This includes familiarity with the values $\cgup$, numbers, and nimbers, as
well as knowledge regarding ordinal sums, and reversibility. We recommend
Lessons in Play \cite{albert2019lessons} as a reference for readers unfamiliar
with these concepts.

In this section, one birthday $3$ game has the focus of our attention. We
denote this game by 
\[
   \Z = \game{\cgup+\cgstar,\cgup,\game{1}{\cgstar,0}}{\game{0,\cgstar}{-1},\cgdown,\cgdown+\cgstar}.
\]
As we mentioned in \cref{sec:search}, there are $19$ values that are not
realisable with at most 7 vertices. So presumably we could have picked any of
those to show $F(3)\geq 8$. The reason we are choosing $\Z$ here is because of
some useful properties that it exhibits, which we explore in the subsequent
lemmas. One particularly useful feature is that $\Z$ is a symmetric form (i.e.\
$\Z=-\Z$).

Now we consider \diplace games with at most one blue vertex.

\begin{lemma}
    \label{Lemma: No blue}
    If $G$ is a \diplace game with at least one red vertex and no blue vertex,
    then $G=k$ for a negative integer $k$.
\end{lemma}

\begin{proof}
    Let $G$ be a \diplace game with no blue vertex. If $G$ has exactly one red
    vertex, then $G=-1$. Assume now that $G$ has $n>1$ vertices. For all red
    vertices $u$ in $G$, the game $G/u$ is a \diplace game with no blue
    vertices and strictly less than $n$ red vertices. If $G/u$ has no vertices,
    then $G/u=0$. Otherwise, $G/u$ is a negative integer by induction. So, Left
    has no options in $G$, and Right's options are to $0$ and/or a negative
    integer. Thus, $G$ is a negative integer.
\end{proof}

\begin{lemma}
    \label{Lemma: Left-win one blue}
    If $G$ is a Left-win \diplace game with exactly one blue vertex, then $G$
    is a positive number and $G\leq 1$.
\end{lemma}

\begin{proof}
    Suppose $G$ is a Left-win \diplace game with exactly one blue vertex: call
    this vertex $v$. As $G$ is Left-win, there must exist a Left option $G^L$
    of $G$ with $G^L\geq0$. It is clear that $G^L = G/v$ and $G^L$ cannot have
    any blue vertices. So \cref{Lemma: No blue} implies $G^L$ has no red
    vertices either. Thus, $G^L$ is the \diplace game with no vertices,
    yielding $G^L=0$. This implies $(v,u)$ is an arc in $G$ for every red
    vertex $u$.

    Since $G$ is Left-win, Right has no winning option. This implies that, for
    all red vertices $u$ in $G$, we must have that $G/u$ is a win either for
    Left or the next player. So, for all red vertices $u$ in $G$, Left has an
    option in $G/u$. This implies that, for all red vertices $u$, there is no
    $(u,v)$ arc in $G$.

    Let $G'$ be the \diplace game formed by removing $v$ (and only $v$) from
    $G$. Then, the reader can easily verify $G=1 \os G'$.

    If $G'$ has no vertices, then clearly $G=1$. Otherwise, $G'$ has no blue
    vertex but at least one red vertex, and so \cref{Lemma: No blue} implies
    $G'=k$ for some negative integer $k$. In this case, the Colon Principle
    implies
    \[
        G=1\os k=\frac{1}{2^{|k|}}.
    \]
    Thus, $G$ is equal to a number satisfying $0<G\leq1$, as desired.
\end{proof}

Next, we consider a lemma that deals with reversibility in \diplace games with
at most 2 blue vertices.

\begin{lemma}
    \label{Lemma: at most 2 blue reversing}
    If $G$ is a \diplace game with at most two blue vertices and $X$ is the
    canonical form of $G$, then for all Left options $X^L$ of $X$, either
    $G/v=X^L$ or else $G/[v,u,w]=X^L$, where $v$ and $w$ are blue vertices, and
    $u$ is a red vertex.
\end{lemma}

\begin{proof}
    Let $G$ be a \diplace game with at most two blue vertices and let $X$ be
    the canonical form of $G$. Let $X^L$ be a fixed but arbitrary Left option
    of $X$. If there exists a blue vertex $v$ such that $G/v = X^L$, then we
    are satisfied. Suppose then for all blue vertices $v$ in $G$, $G/v \neq
    X^L$. It follows that $X^L$ is not a Left option of $G$.

    Since $X$ is the canonical form of $G$, the literal form of $X$ can be
    obtained from the literal form of $G$ by removing dominated options and
    reversing reversible options. In particular, by removing dominated Left
    options, and through reversing Left options, $G$ is transformed into a
    position, say $G'$, with a Left option equal to $X^L$. Removing a dominated
    option never gives a new option to Left; it follows that the option equal
    to $X^L$ in $L(G')$ is obtained through reversing a Left option of $G$ one
    or more times.

    Thus, there exists a sequence $G = G_0,G_1, G_2, \dots, G_{2k-1} = X^L$
    where $G_i \in R(G_{i-1})$ for odd $i$, and $G_j \in L(G_{j-1})$ for even
    $j\geq 2$. By our assumption that $G$ contains at most $2$ blue vertices,
    if $G = G_0,G_1, G_2, \dots, G_{2k-1} = X^L$, then $k\leq 2$. This is
    because for all even $j\geq 2$, $G_j$ is obtained from $G_{j-1}$ by
    deleting a blue vertex. Here every $G_t$ is a subgraph of $G_{t-1}$. So all
    $G_t$ are subgraphs of $G = G_0$.

    Therefore, either $G$ has a Left option to $X^L$ which can be written $G/v
    = X^L$ for a blue vertex $v$, or there exists a sequence $G = G_0 , G_1,
    G_2, G_3 = X^L$ where $G_1 \in L(G_0), G_2\in R(G_1), G_3 \in L(G_2)$. Let
    $G = G_0 , G_1, G_2, G_3 = X^L$ be such a sequence, then there exist blue
    vertices $v$ and $w$, and a red vertex $u$ such that $G_1 = G/v$, $G_2 =
    G/[v,u]$, and $X^L = G_3  = G/[v,u,w]$. This concludes the proof.
\end{proof}

We now study the Left options of $\Z$.

\begin{lemma}
    \label{Lemma: Other 2 blue}
    If $G$ is a \diplace game and $G = \game{1}{0,\cgstar}$, then $G$ contains
    at least two blue vertices.
\end{lemma}

\begin{proof}
    Let $G$ be a \diplace game equal to $\game{1}{0,\cgstar}$. If $G$ has an
    option to $1$, call it $G/v$, then $G/v$ contains a blue vertex by
    Lemma~\ref{Lemma: No blue}. If $G/v = 1$ is a Left option of $G$, then $v$
    is a blue vertex that is not in $G/v$. Hence, if  $G$ has a Left option to
    $1$, then $G$ has at least two blue vertices.

    Otherwise, $G$ has no option to $1$. Then Lemma~\ref{Lemma: at most 2 blue
    reversing} implies there exist blue vertices $v$ and $w$, and a red vertex
    $u$, such that $G/[v,u,w] = 1$. Again, this implies $G$ has at least two
    blue vertices. Therefore, if $G$ is a \diplace game equal to
    $\game{1}{0,\cgstar}$, then $G$ contains at least two blue vertices.
\end{proof}

\begin{lemma}
    \label{Lemma: Up 2 blue}
    If $G$ is a \diplace game and $G = \cgup$, then $G$ contains at least two
    blue vertices.
\end{lemma}

\begin{proof}
    Let $G$ be a \diplace game equal to $\cgup$. Then, $G$ is Left-win and $G$
    is not a number. As $G$ is Left-win, Lemma~\ref{Lemma: No blue} implies $G$
    requires at least one blue vertex. As $G$ is Left-win and not a number,
    Lemma~\ref{Lemma: Left-win one blue} implies $G$ does not have exactly one
    blue vertex. Thus, $G$ has at least two blue vertices.
\end{proof}

\begin{lemma}
    \label{Lemma: Up Star 2 blue}
    If $G$ is a \diplace game and $G = \cgup + \cgstar$, then $G$ contains at
    least two blue vertices. If $G$ additionally contains exactly two blue
    vertices, then the blue vertices of $G$ do not form an independent set.
\end{lemma}

\begin{proof}
    Let $G$ be a \diplace game equal $\cgup + \cgstar$. Note that
    $\game{0,\cgstar}{0}$ is the canonical form of $\cgup + \cgstar$. We begin
    by proving $G$ has at least $2$ blue vertices.

    If $G$ has a Left option to $\cgstar$, call it $G/v$, then $G/v$ contains a
    blue vertex by Lemma~\ref{Lemma: No blue}. If $G/v = \cgstar$ is a Left
    option of $G$, then $v$ is a blue vertex that is not in $G/v$. Hence, if
    $G$ has an option to $\cgstar$, then $G$ has at least two blue vertices.

    Otherwise, $G$ has no option to equal to $\cgstar$. Then Lemma~\ref{Lemma:
    at most 2 blue reversing} implies there exists a blue vertices $v$ and $w$,
    and a red vertex $u$, such that $G/[v,u,w] = \cgstar$. Since $\cgstar$ is
    not a negative integer Lemma~\ref{Lemma: No blue} implies $G/[v,u,w]$ has a
    blue vertex. Thus, if Left has no option equal to $\cgstar$ in $G$, then
    $G$ has at least three blue vertices. Therefore, if  $G$ is a \diplace game
    equal to $\cgup + \cgstar$, then $G$ contains at least two blue vertices.

    Now suppose $H$ is a \diplace game such that $H$ contains exactly two blue
    vertices and $H = \cgup+\cgstar$. Then, from the previous paragraph, there
    exists a blue vertex $v$ in $H$ such that $H/v = \cgstar$. Since $H = \cgup
    + \cgstar$, $\game{0,\cgstar}{0}$ is the canonical form of $H$. Suppose for
    the sake of contradiction that the blue vertices of $H$ form an independent
    set. We consider the following two cases:
    \begin{enumerate}
        \item There exists a blue vertex $w$ in $H$ such that $H/w = 0$.

            Since the blue vertices of $G$ form an independent set, $v$ is a
            vertex in $H/w$. Hence, $H/w$ has exactly one blue vertex. As $H/w
            = 0$ Left has no winning move in $H/w$. Then, there exists a red
            vertex $u$ in $H/w$ such that $(v,u)$ is not an arc in $H/w$.
            Otherwise, $H/[w,v]$ is a zero vertex graph, implying that Left has
            a winning option in $H/w$. Furthermore, since $u$ is a vertex in
            $H/w$, $(w,u)$ is not an arc in $H$. Thus, $(v,u)$ and $(w,u)$ are
            not arcs in $H$.

            Now consider $H/v$ which equals $\cgstar$. Since $(v,u)$ is not an
            arc in $H$, $u$ is a red vertex in $H/v$. Given $(w,u)$ is not an
            arc in $H$, we note that $(w,u)$ is not an arc in $H/v$. This
            implies $u$ is a vertex in $H/[v,w]$. Hence, as $u$ is a red vertex
            in $H/[v,w]$, while  $H/[v,w]$ contains no blue vertices by our
            assumption $H$ contains exactly two blue vertices. Since $H/[v,w]$
            contains at least one red vertex, $u$, and no blue vertices,
            Lemma~\ref{Lemma: No blue} implies $H/[v,w] < 0$.

            But this is a contradiction, since $H/v = \cgstar$, while $H/v$ has
            exactly one blue vertex $w$ implies that $H/[v,w] = 0$. So we have
            shown $0 = H/[v,w] < 0$ which is absurd. 

            It follows that if Left has an option to $0$ in $H$, then we
            contradict that $H = \cgup +\cgstar$, or we contradict that $H$ has
            exactly two blue vertices, or we contradict that the blue vertices
            of $H$ form an independent set.

        \item No Left option of $H$ is equal to $0$.

            Then \cref{Lemma: at most 2 blue reversing} implies $H$ must have
            $0$ as a reversible Left option. Let $w$ be the second blue vertex
            in $H$, recalling we have assumed $H/v = \cgstar$. As $0$ is not an
            option of any Right option of $\cgstar$ and $H/v = \cgstar$, there
            exists a red vertex $u \notin N^+(w)$ such that $0$ is a Left
            option of $H/[w,u]$ and $H/[w,u] \leq H$. Let $u$ be such a vertex.

            Since $0$ is a Left option of $H/[w,u]$, the game $H/[w,u]$ must
            have a blue vertex. This blue vertex must be $v$. Hence, $H/[w,u,v]
            = 0$. Then \cref{Lemma: No blue} and the fact that $H/[w,u]$ has
            only one blue vertex, implies that for every red vertex $z$ in
            $H/[w,u]$, $(v,z)$ is an arc in $H/[w,u]$.

            Since $H/[w,u] \leq H = \cgup + \cgstar$, where $\cgup + \cgstar$
            is next player win, Right has a winning option in $H/[w,u]$. If $x$
            is a red vertex such that $H/[w,u, x] \leq 0$, then $(x,v)$ is an
            arc in $H/[w,u]$. Otherwise, the vertex $v$ exists in $H/[w,u, x]$
            implying Left can win moving first in $H/[w,u, x]$ by delete $v$
            and all remaining red vertices, contradicting that $H/[w,u, x] \leq
            0$. Since Right has a winning move in $H/[w,u]$ there exists a red
            vertex $x$ such that $H/[w,u, x] \leq 0$. Let $x$ be such a red
            vertex.

            Then, $H/[w,x]$ has no blue vertices, since $v$ is the only blue
            vertex in $H/w$ and $(x,v)$ is an arc in $H$. Hence, \cref{Lemma:
            No blue} implies $H/[w,x]$ is equal to $0$ or a negative integer.
            It follows that Right-wins $H/w$ moving first.

            But this implies Left's only options in $H$ are to $H/v = \cgstar$
            and to $H/w$ where Right-wins moving first in $H/w$. Thus, Left
            loses moving first in $H$. This is a contradiction as we have
            assumed $H = \cgup + \cgstar$, which is next player win.
    \end{enumerate}

    Therefore, we have shown that if $H$ is a \diplace game satisfying that $H$
    contains exactly two blue vertices, $H = \cgup+\cgstar$, and the blue
    vertices of $H$ form an independent set, then this leads to a
    contradiction. It follows no such $H$ exists. This concludes the proof.
\end{proof}

\begin{theorem}
    \label{Thm: Day 3, 8 vertices, lower}
    There exists a game $X$ with birthday $3$ such that $X= -X$ and
    $f(X)\geq8$.
\end{theorem}

\begin{proof}
    Recall that $\Z$ is
    $\game{\cgup+\cgstar,\cgup,\game{1}{\cgstar,0}}{\game{0,\cgstar}{-1},\cgdown,\cgdown+\cgstar}$.
    It is easy to verify that $\Z$ has birthday $3$, is in canonical form, and
    is a symmetric form (i.e.\ $\Z=-\Z$). We show that $f(\Z)\geq 8$.

    Suppose for the sake of contradiction $f(\Z) < 8$. Then, there exists a
    \diplace game $H = \Z$ with at most $7$ vertices. Since $H$ has at most $7$
    vertices, there cannot be at least $4$ blue vertices and at least $4$ red
    vertices in $H$. If there are at most $3$ blue vertices in $H$, let $G
    \cong H = \Z$. If there are at most $3$ red vertices in $H$, let $G \cong -
    H = -\Z = \Z$. Thus, \diplace game $G$ has at most $3$ blue vertices.

    If $G$ has at most $2$ blue vertices, then Left cannot move in $G$ to
    $\cgup$ by Lemma~\ref{Lemma: Up 2 blue}. Moreover, $\cgup$ is not be a Left
    option of $G$ that can be obtained by reversibility since any game
    $G/[v,u,w]$ where $v,w$ are blue vertices and $u$ is red, would have less
    than $2$ blue vertices. As this contradicts $G = \Z$, we suppose $G$ has
    exactly $3$ blue vertices.

    Every Left option of $G$ obtained from reversibility is of the form $G' =
    G/[v,u,w]$ where $v$ and $w$ are blue vertices and $u$ is a red vertex, or
    of the form $G' = G/[v,u,w,x,y]$ where $v,w,y$ are blue vertices and $u,x$
    are red vertices. Since $G$ has exactly three blue vertices, any such $G'$
    has at most one blue vertex. So by Lemma~\ref{Lemma: Other 2 blue} $G$
    cannot have $\game{1}{0,\cgstar}$ as a Left option from reversibility.
    Similarly, by Lemma~\ref{Lemma: Up 2 blue} $G$ cannot have $\cgup$ as a
    Left option from reversibility, and by Lemma~\ref{Lemma: Up Star 2 blue}
    $G$ cannot have $\cgup+\cgstar$ as a Left option from reversibility.

    Then, $G$ cannot obtain any of the Left options of $\Z$ via reversibility.
    It follows every Left option of $G$ is equal to a Left option of $\Z$.
    Hence, there are blue vertices $v_1,v_2,v_3$ in $G$ where
    \begin{itemize}
        \item $G/v_1 =  \game{1}{0,\cgstar}$, and
        \item $G/v_2 =  \cgup$, and
        \item $G/v_3 =  \cgup+\cgstar$.
    \end{itemize}
    Since each of $\game{1}{0,\cgstar}, \cgup+\cgstar,$ and $\cgup$ require at
    least $2$ blue vertices to be realized as \diplace games, $G$ having
    exactly $3$ blue vertices implies $\{v_1,v_2,v_3\}$ forms an independent
    set.

    Observe that $G/v_3$ is a subgraph of $G$ with exactly two blue vertices.
    Furthermore, we have assumed without loss of generality that $G/v_3 =
    \cgup+\cgstar$. Since $\{v_1,v_2,v_3\}$ is an independent set in $G$,
    $\{v_1,v_3\}$ is an independent set in $G/v_3$. But this implies $G/v_3$ is
    a \diplace game equal to $\cgup+\cgstar$, with exactly two blue vertices,
    where the blue vertices in $G/v_3$ form an independent set. This is a
    contradiction given that Lemma~\ref{Lemma: Up Star 2 blue} states no such
    \diplace game exists.

    Therefore, we have contradicted that there exists a \diplace game equal to
    $\Z$ on at most $7$ vertices. It follows that $f(\Z)\geq 8$ as desired.
\end{proof}

\begin{figure}[ht]
    \centering
    \scalebox{0.7}{
    \includegraphics{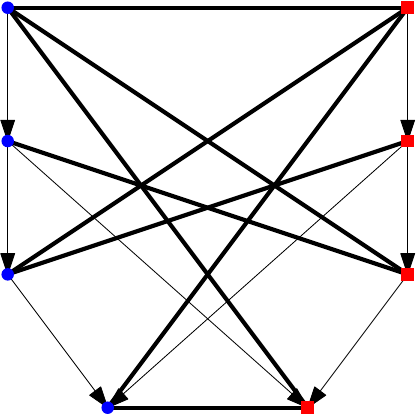}}
    \caption{An $8$-vertex \diplace game equal to $\Z$.}
    \label{fig:X}
\end{figure}

To see a digraph that shows $f(\Z) = 8$ we provide the reader
Figure~\ref{fig:X}. We additionally observe that there is an automorphism of
the underlying digraph of $G$ that switches red and blue vertices. That is, $G$
and $-G$ are isomorphic as $2$-coloured digraphs. It is immediate that this is
only possible if $G=-G$.

\section{Games Born Day 4 and Day 5}
\label{sec:better-bounds}

Now that we have proven $F(3) = 8$ we turn our attention to the values of
$F(4)$ and $F(5)$. Unfortunately, we are not able to prove an exact value for
either constant. However, we are able to improve the bounds previously given in
\cite{clow2024digraph}. To see the bounds from \cite{clow2024digraph} consider
Table~\ref{tab:old F(b)}.

\begin{table}[ht]
    \begin{center}
        \begin{tabular}{|c |c | c |} 
            \hline
            $b$ & $ \leq F(b) $ & $ F(b) \leq $  \\ [0.5ex] 
            \hline\hline
            $0$ & $0$ & $0$ \\ 
            \hline
            $1$ & $2$ & $2$\\
            \hline
            $2$ & $4$ & $4$\\
            \hline
            $3$ & $5$ & $74$  \\
            \hline
            $4$ & $6$ & \num{13439} \\ 
            \hline
            $5$ & $7$ & $1.08 \cdot 10^{189}$\\
            \hline
        \end{tabular}
    \end{center}
    \caption{Bounds on $F(b)$ for $b\leq 5$ from \cite{clow2024digraph}.}
    \label{tab:old F(b)}
\end{table}

We improve the bounds in Table~\ref{tab:old F(b)} by leveraging a number of
existing bounds in the literature, along with the fact that we have shown $F(3)
= 8$. Our upper bounds rely on a bound by Suetsugu \cite{Suets2022b} on the
number of game values with birthday at most $4$, and a recursive bound on
$F(b)$ given in \cite{clow2024digraph}. Meanwhile, our lower bounds combine the
same work by Suetsugu \cite{Suets2022b}, earlier lower bounds on the number of
game values with birthday at most $b$ by Wolfe and Fraser
\cite{wolfe2004counting}, and counting the number of isomorphism classes of
$2$-coloured digraphs on $n$ vertices.

We let $g(b)$ denote the number of game values, that is equivalence classes of
games, that exist with birthday at most $b$. The currently best bounds on
$g(4)$ are included below.

\begin{lemma}[{\cite[Corollary 2, Theorem 2]{Suets2022b}}]
    \label{Lemma: Koki g(4) lower}
    $2^{94} < g(4) \leq 4 \cdot 10^{184}$.
\end{lemma}

Next, we give a useful lower bound on $g(b+1)$ in terms of $g(b)$ and $g(b-1)$.

\begin{lemma}[{\cite[Theorem 8]{wolfe2004counting}}]
    \label{Lemma: g(n) lower}
    For all $b\geq 0$,
    \[
        g(b+1)\geq \Big(8g(b-1)-4\Big) \Big( 2^{\frac{g(b)-2}{2g(b-1)-1}-1} -
        1\Big).
    \]
\end{lemma}

Combining these results, and recalling that $g(3) = 1474$, we give a lower
bound for $g(5)$.

\begin{lemma}
    \label{Lemma: Lower bound for g(5)}
    $g(5) > 2^{6.7(10^{24})}$.
\end{lemma}

\begin{proof}
    Recall that $g(3) = 1474$ and by Lemma~\ref{Lemma: Koki g(4) lower} that
    $g(4) > 2^{94}$. Thus, Lemma~\ref{Lemma: g(n) lower} implies 
    \begin{align*}
        g(5) & \geq \Big(8g(3)-4\Big) \Big( 2^{\frac{g(4)-2}{2g(3)-1}-1} -
        1\Big) \\
        & > 11788 \Big(2^{\frac{2^{94}-2}{2947}-1} - 1\Big) \\
        & \geq 2^{6.7(10^{24})}.
    \end{align*}
    This concludes the proof.
\end{proof}

For a non-negative integer $n$ let $D(n)$ be the number of $2$-coloured
digraphs on $n$ unlabelled vertices with no more than one arc with the same
start and end vertex. As noted in \cref{sec:search}, there is a bijection
between this set of $n$-vertex digraphs and the set of $n$ vertex \diplace
games. This means that $D(n)$ counts the number of $n$-vertex \diplace games,
up to graph isomorphism. The list of small values for $D(n)$ is OEIS sequence
A000595. See Table~\ref{tab:D(n)}.

\begin{table}[ht]
    \begin{center}
        \begin{tabular}{|r |r |} 
            \hline
            $n$ & $ D(n) $ \\ [0.5ex] 
            \hline\hline
            $0$ & $1$ \\ 
            \hline
            $1$ & $2$ \\
            \hline
            $2$ & $10$ \\
            \hline
            $3$ & $104$ \\
            \hline
            $4$ & $\num{3044}$ \\ 
            \hline
            $5$ & \num{291968} \\
            \hline
            $6$ & \num{96928992} \\ 
            \hline
            $7$ &  \num{112282908928} \\
            \hline
            $8$ &  \num{458297100061728} \\
            \hline
            $9$ & \num{6666621572153927936} \\
            \hline
            $10$ & \num{349390545493499839161856}\\ 
            \hline
            $11$ & \num{66603421985078180758538636288} \\
            \hline
            $12$ & \num{46557456482586989066031126651104256} \\
            \hline
            $13$ & \num{120168591267113007604119117625289606148096}\\
            \hline
            $14$ & \num{1152050155760474157553893461743236772303142428672} \\
            \hline
        \end{tabular}
    \end{center}
    \caption{A list of small values of $D(n)$.}
    \label{tab:D(n)}
\end{table}

Additionally, an exact formula for $D(n)$ was given by Davis
\cite{davis1953number}. Using this formula, McIlroy
\cite{mcilroy1955calculation} showed that that the ratio of $D(n)$ and
$2^{n^2}/n!$ tends to $1$ as $n$ tends to infinity. When we consider day $5$
games, taking the following upper bound on $D(n)$ is sufficient for our
purposes.

\begin{lemma}
    \label{Lemma: Upper Bounds on D(n)}
    For all $n \geq 0$, $D(n) \leq 2^{n^2}$.
\end{lemma}

\begin{proof}
    Trivially, for all $n\geq 0$, $D(n)$ is bounded above by the number of
    labelled $2$-vertex coloured digraphs with no more than one arc with the
    same start and end vertex. We claim that for all $n\geq 0$, there are
    exactly $2^{n^2}$ digraphs of this type.

    Notice there are $n(n-1)$ ordered pairs $(u,v)$ for $u,v \in \{1,\dots,
    n\}$. Each ordered pair $(u,v)$ corresponds to a possible arc. Hence, there
    are $2^{n(n-1)}$ labelled $1$-vertex colour digraphs with no more than one
    arc with the same start and end vertex. From here it is easy to see that
    there are $2^{n(n-1)} \cdot 2^n = 2^{n^2}$ digraphs with no more than one
    arc with the same start and end vertex whose vertices are $2$-coloured.
    Therefore, $D(n) \leq 2^{n^2}$ as desired.
\end{proof}

For games $X$ and $Y$, we say $X \leq Y$ if Right wins $X-Y$ moving second.
Under this partial order, $X \leq Y$ and $Y \leq X$ if and only if $X=Y$. For
more on the partial order of games, see \cite{siegel2013combinatorial}.

Next, we recall an upper bound on $F(b+1)$ which depends on the size of a
largest anti-chain of games with birthday at most $b$, and $F(b)$. This upper
bound is taken from \cite{clow2024digraph}.

\begin{lemma}[{\cite[Lemma 5.1]{clow2024digraph}}]
    \label{Lemma: F(b) upper bound in terms of b and a(b)}
    Let $b\geq 0$. Then 
    \[
        F(b+1) \leq 2a(b)(F(b)+b+1) + 5b + 8
    \]
    where $a(b)$ is the order of a largest anti-chain of games born by day $b$.
\end{lemma}

In order to apply the bound in Lemma~\ref{Lemma: F(b) upper bound in terms of b
and a(b)} we require information about $a(b)$. The size of a largest anti-chin
of games born by $3$ was proven to be $86$ in \cite{Suets2022b}. We take $g(4)$
as a trivial upper bound for $a(4)$.

\begin{lemma}[{\cite[Lemma 3, Theorem 2]{Suets2022b}}]
    \label{Lemma: antichain bounds}
    Letting $a(b)$ be the order of a largest anti-chain of games born by day
    $b$, $a(3) = 86$ and $a(4) \leq 10^{184}$.
\end{lemma}

We are now prepared to prove the main result of this section.

\begin{theorem}\label{Thm: F(4),F(5) bounds}
    The function $F$ satisfies
    \begin{itemize}
        \item $11 \leq F(4) \leq 2087$, and
        \item $2.58(10^{12}) <  F(5) < 10^{187.63}$.
    \end{itemize}
\end{theorem}

\begin{proof}
    We begin by proving the stated lower bounds. Since each $2$-coloured
    digraph corresponds to exactly one \diplace game, $$ \sum_{n=0}^{F(b)} D(n)
    \geq g(b) $$ with equality if and only if each every \diplace game on at
    most $n$ vertices has a distinct value, which is false for \diplace games
    on two vertices. Hence, Lemma~\ref{Lemma: Koki g(4) lower} and
    Lemma~\ref{Lemma: Lower bound for g(5)} imply
    \begin{align*}
        \sum_{n=0}^{F(4)} D(n) > 2^{94} \geq 10^{28.2} & &\text{ and } & &
        \sum_{i=0}^{F(5)} D(n) > 2^{6.7(10^{24})}
    \end{align*}
    respectively. Given Table~\ref{tab:D(n)} the reader can easily verify $11$
    is the smallest integer such that
    \[
     \sum_{n=0}^{11} D(n) > 10^{28.2}.
    \]
    Thus, $F(4)\geq 11$ as required. By Lemma~\ref{Lemma: Upper Bounds on
    D(n)}, for all $n\geq 0$, $D(n) \leq 2^{n^2}$. Hence, 
    \begin{align*}
        \sum_{n=0}^{F(5)} 2^{n^2} & > 2^{6.7(10^{24})}.
    \end{align*}
    Inspecting this inequality and taking logarithms gives 
    \begin{align*}
        F(5)^2  + 1 & = \Bigg\lceil \log_2\Bigg(\sum_{i=0}^{F(5)} 2^{n^2}
        \Bigg) \Bigg\rceil \\
        & > 6.7(10^{24}).
    \end{align*}
    Hence, $F(5) \geq 2.58(10^{12})$ as required.

    Next we derive upper bounds on $F(4)$ and $F(5)$. To do this we rely on
    Lemma~\ref{Lemma: F(b) upper bound in terms of b and a(b)} and
    Lemma~\ref{Lemma: antichain bounds}. Applying both of these Lemmas, in
    addition to $F(3) = 8$ per Theorem~\ref{Thm: F(3) = 8}, gives
    \begin{align*}
        F(4) &\leq 2(86)(8+3+1)+5(3)+8 \\
             & = 2087
    \end{align*}
    as desired. Similarity, by Lemma~\ref{Lemma: F(b) upper bound in terms of b
    and a(b)} and Lemma~\ref{Lemma: antichain bounds}, 
    \begin{align*}
        F(5) &\leq 2(10^{184})(2087+4+1)+5(4)+8 \\
             & = 4184(10^{184}) + 28 \\
             & < 4185(10^{184}) \\
             & < 10^{187.63}.
    \end{align*}
    This completes the proof.
\end{proof}

\section{Final remarks}
\label{sec:final}

Here we discuss some questions for future work. Of course, the most obvious
problem is, can we improve the bounds on $F(4)$ and $F(5)$? Given $F(5) >
10^{12}$ we are particularity interested in whether $F(4)$ is small, say at
most $100$, as it would be interesting for $F(b)$ to jump from having two
digits to at least $12$ in only one step. We pose the following easier problem.

\begin{problem}
    Determine if $F(4)$ is larger or smaller than $100$.
\end{problem}

Next we recall Problem~6.3 from \cite{clow2024digraph}, which asks to find a
family of games $S$, such that there exists a constant $c>0$, where if $S_b$ is
the set of all games in $S$ born by day $b$, then
\[
    \min_{X\in S_b} f(X) \geq cF(b).
\]
Given the properties of the game $\Z$ used in the proof of Theorem~\ref{Thm:
Day 3, 8 vertices, lower}, we propose that some subset of the games $X$ such
that $X = - X$ could be such a family. More strongly we conjecture the
following.

\begin{conjecture}
    \label{Conj: X = -X is hard to build}
    For all integers $b\geq 0$, there exists a game $X$ with birthday $b$ such
    that $X = - X$ and 
    \[
        f(X) = F(b).
    \]
\end{conjecture}

This conjecture is true for all $b\leq 3$. To see this, take the games
$0,\cgstar,\game{1}{-1},$ and $\Z$, for the cases $b = 0,1,2,$ and $3$
respectively. 

Next, we reflect on the fact that to prove Theorem~\ref{Thm: Day 3, 8 vertices,
lower} we relied on showing certain values do not exist in \diplace if there is
exactly one blue vertex. Could this class be characterised?

\begin{problem}
    Characterise the values found in \diplace with exactly one blue vertex.
\end{problem}

One might also ask, what values do not appear in \diplace games with at most
$k$ blue vertices? Is such a result useful in proving a better lower bound for
$F(4)$?

\bibliographystyle{plainurl}
\bibliography{bib}

\begin{thebibliography}{10}

\bibitem{albert2019lessons}
Michael Albert, Richard Nowakowski, and David Wolfe.
\newblock {\em Lessons in play}.
\newblock Taylor \& Francis, London, England, 2 edition, January 2023.

\bibitem{BerleCG1982}
Elwyn~R. Berlekamp, John~H. Conway, and Richard~K. Guy.
\newblock {\em Winning ways for your mathematical plays. {V}ol. 1}.
\newblock A K Peters, Ltd., Natick, MA, second edition, 2001.

\bibitem{bodlaender2002kayles}
Hans~L. Bodlaender and Dieter Kratsch.
\newblock Kayles and nimbers.
\newblock {\em J. Algorithms}, 43(1):106--119, 2002.
\newblock \href {https://doi.org/10.1006/jagm.2002.1215} {\path{doi:10.1006/jagm.2002.1215}}.

\bibitem{brown2019note}
Jason~I. Brown, Danielle Cox, Andrew~H. Hoefel, Neil McKay, Rebecca Milley, Richard~J. Nowakowski, and Angela~A. Siegel.
\newblock A note on polynomial profiles of placement games.
\newblock In {\em Games of no chance 5}, volume~70 of {\em Math. Sci. Res. Inst. Publ.}, pages 243--258. Cambridge Univ. Press, Cambridge, 2019.

\bibitem{clow2023ordinal}
Alexander Clow and Neil McKay.
\newblock Ordinal sums of numbers, 2023.
\newblock URL: \url{https://arxiv.org/abs/2305.16516}, \href {https://arxiv.org/abs/2305.16516} {\path{arXiv:2305.16516}}.

\bibitem{clow2024digraph}
Alexander Clow and Neil~A McKay.
\newblock Digraph placement games, 2025.
\newblock URL: \url{https://arxiv.org/abs/2407.12219}, \href {https://arxiv.org/abs/2407.12219} {\path{arXiv:2407.12219}}.

\bibitem{Conwa1976a}
J.~H. Conway.
\newblock {\em On numbers and games}.
\newblock A K Peters, Ltd., Natick, MA, second edition, 2001.

\bibitem{davies:gemau}
Alfie Davies.
\newblock \texttt{gemau}, September 2024.
\newblock URL: \url{https://github.com/alfiemd/gemau}.

\bibitem{davis1953number}
Robert~L. Davis.
\newblock The number of structures of finite relations.
\newblock {\em Proc. Amer. Math. Soc.}, 4:486--495, 1953.
\newblock \href {https://doi.org/10.2307/2032159} {\path{doi:10.2307/2032159}}.

\bibitem{mcilroy1955calculation}
M.~D. McIlroy.
\newblock Calculation of numbers of structures of relations on finite sets.
\newblock Technical Report~17, Massachusetts Institute of Technology, 1955.

\bibitem{MCKAY201494}
Brendan~D. McKay and Adolfo Piperno.
\newblock Practical graph isomorphism,~ii.
\newblock {\em Journal of Symbolic Computation}, 60:94--112, 2014.
\newblock URL: \url{https://www.sciencedirect.com/science/article/pii/S0747717113001193}, \href {https://doi.org/10.1016/j.jsc.2013.09.003} {\path{doi:10.1016/j.jsc.2013.09.003}}.

\bibitem{oeis}
{OEIS Foundation Inc.}
\newblock The {O}n-{L}ine {E}ncyclopedia of {I}nteger {S}equences, 2025.
\newblock Published electronically at \url{http://oeis.org}.

\bibitem{schaefer1978complexity}
Thomas~J. Schaefer.
\newblock On the complexity of some two-person perfect-information games.
\newblock {\em J. Comput. System Sci.}, 16(2):185--225, 1978.
\newblock \href {https://doi.org/10.1016/0022-0000(78)90045-4} {\path{doi:10.1016/0022-0000(78)90045-4}}.

\bibitem{siegel2013combinatorial}
Aaron~N. Siegel.
\newblock {\em Combinatorial game theory}, volume 146 of {\em Graduate Studies in Mathematics}.
\newblock American Mathematical Society, Providence, RI, 2013.
\newblock \href {https://doi.org/10.1090/gsm/146} {\path{doi:10.1090/gsm/146}}.

\bibitem{Suets2022b}
Koki Suetsugu.
\newblock Improving upper and lower bounds of the number of games born by day 4, 2024.
\newblock To appear in \emph{Games of no chance 6}, volume 71 of \emph{Math. Sci. Res. Inst. Publ.}, Cambridge Univ. Press.
\newblock URL: \url{https://arxiv.org/abs/2208.13403v2}, \href {https://arxiv.org/abs/2208.13403} {\path{arXiv:2208.13403}}.

\bibitem{wolfe2004counting}
David Wolfe and William Fraser.
\newblock Counting the number of games.
\newblock {\em Theoretical Computer Science}, 313(3):527--532, 2004.
\newblock Algorithmic Combinatorial Game Theory.
\newblock URL: \url{https://www.sciencedirect.com/science/article/pii/S0304397503005991}, \href {https://doi.org/10.1016/j.tcs.2003.05.001} {\path{doi:10.1016/j.tcs.2003.05.001}}.

\end{thebibliography}

\appendix
\newpage
\section*{Day 3 Values not found on at most 7 vertices}

The following are the $19$ values born on day $3$ we did not find in our search
of $7$-vertex digraphs.

\begin{enumerate}
    \item
        $\game{\cgdown,\cgdown+\cgstar,\pm1}{-2}$
    \item
        $\game{2}{\pm1,\cgup+\cgstar,\cgup}$

    \item
        $\game{\cgdown,\cgdown+\cgstar,\pm1}{-1,-1+\cgstar}$
    \item
        $\game{1+\cgstar,1}{\pm1,\cgup+\cgstar,\cgup}$

    \item
        $\game{\cgstar,\cgstar2,0}{\game{\cgstar,0}{-1},\cgdown,\cgdown+\cgstar}$
    \item
        $\game{\cgup+\cgstar,\cgup,\game{1}{0,\cgstar}}{0,\cgstar2,\cgstar}$

    \item
        $\game{0,\cgup+\cgstar,\game{1}{\cgstar,0}}{\game{\cgstar,0}{-1},\cgdown,\cgdown+\cgstar}$
    \item
        $\game{\cgup+\cgstar,\cgup,\game{1}{0,\cgstar}}{\game{0,\cgstar}{-1},\cgdown+\cgstar,0}$ 

    \item
        $\game{0,\cgup+\cgstar,\game{1}{\cgstar,0}}{\cgstar,\game{\cgstar,0}{-1},\cgdown}$
    \item
        $\game{\cgup,\game{1}{0,\cgstar},\cgstar}{\game{0,\cgstar}{-1},\cgdown+\cgstar,0}$

    \item
        $\game{\cgstar,\cgup,\game{1}{\cgstar,0}}{\game{\cgstar,0}{-1},\cgdown,\cgdown+\cgstar}$
    \item
        $\game{\cgup+\cgstar,\cgup,\game{1}{0,\cgstar}}{\game{0,\cgstar}{-1},\cgdown,\cgstar}$
    \item
        $\game{\cgup+\cgstar,\cgup,\game{1}{\cgstar,0}}{\cgdown,\game{0}{-1}}$
    \item
        $\game{\game{1}{0},\cgup}{\game{0,\cgstar}{-1},\cgdown,\cgdown+\cgstar}$

    \item
        $\game{\cgup+\cgstar,\cgup,\game{1}{\cgstar,0}}{\game{\cgstar}{-1},\game{-1}{0},\game{0}{-1}}$
    \item
        $\game{\game{1}{0},\game{0}{1},\game{1}{\cgstar}}{\game{0,\cgstar}{-1},\cgdown,\cgdown+\cgstar}$

    \item
        $\game{\cgstar,\cgup,\game{1}{\cgstar,0}}{\game{0,\cgstar}{-1},\cgdown,\cgstar}$

    \item
        $\game{0,\cgup+\cgstar,\game{1}{\cgstar,0}}{\game{0,\cgstar}{-1},\cgdown+\cgstar,0}$

    \item
        $\game{\cgup+\cgstar,\cgup,\game{1}{\cgstar,0}}{\game{0,\cgstar}{-1},\cgdown,\cgdown+\cgstar}$
\end{enumerate}

\end{document}